\documentclass{amsart}  

\usepackage[usenames, dvipsnames]{color}
\usepackage{amsfonts}
\usepackage{amsthm}
\usepackage{amsmath, mathtools}
\usepackage{amsfonts}
\usepackage{latexsym}
\usepackage{amssymb}
\usepackage{amscd}
\usepackage[latin1]{inputenc}
\usepackage{verbatim}
\usepackage{enumerate}
\usepackage{graphicx}
\usepackage{dsfont}
\usepackage{bbm}

\newtheorem{theorem}{Theorem}[section]
\newtheorem{lemma}[theorem]{Lemma}
\newtheorem{proposition}[theorem]{Proposition}
\newtheorem{corollary}[theorem]{Corollary}
\newtheorem{example}[theorem]{Example}

\newtheorem{conjecture}{Conjecture}[section]

\newtheorem{question}{Question}[section]
\theoremstyle{definition}
\DeclareMathOperator{\mrr+}{mr_{+}}

\DeclareMathOperator{\mr}{mr}
\DeclareMathOperator{\nul}{null}

\DeclareMathOperator{\rk}{rank}

\newcommand\Tstrut{\rule{0pt}{2.6ex}}         
\newcommand\Bstrut{\rule[-0.9ex]{0pt}{0pt}}   

\title[Achievable Multiplicity partitions in the IEVP of a graph]{Achievable Multiplicity partitions in the inverse eigenvalue problem of a graph}

\author[M.~Adm]{Mohammad Adm\textsuperscript{1,2}}
\thanks{\textsuperscript{1}Department of Applied Mathematics and Physics, Palestine Polytechnic University, Hebron, Palestine}
\thanks{\textsuperscript{2}Department of Mathematics and Statistics,  University of Konstanz, Konstanz, Germany}
\thanks{\textsuperscript{3}Department of Mathematics and Statistics,  University of Regina, Regina, SK, S4S 0A2, Canada}
\author[S.~Fallat]{Shaun Fallat\textsuperscript{3}}
\author[K.~Meagher]{Karen Meagher\textsuperscript{3}}
\author[S.~Nasserasr]{Shahla Nasserasr\textsuperscript{3,4}}
\thanks{\textsuperscript{4}Department of Mathematics and Computer Science, Brandon University,
Brandon, MB R7A 6A9, Canada}
\author[S.~Plosker]{Sarah Plosker\textsuperscript{3,4}}
\author[B.~Yang]{Boting Yang\textsuperscript{5}}
\thanks{\textsuperscript{5}Department of Computer Science,  University of Regina, Regina, SK, S4S 0A2, Canada}

\date{\today}

 		
\begin{document}

\begin{abstract}  
Associated to a graph $G$ is a set $\mathcal{S}(G)$ of all real-valued symmetric matrices whose off-diagonal entries are nonzero precisely when the corresponding vertices of the graph are adjacent, and the diagonal entries are free to be chosen. 
If $G$ has $n$ vertices, then the multiplicities of the eigenvalues of any matrix in $\mathcal{S}(G)$ partition $n$; this is called a multiplicity partition. 

We study graphs for which a multiplicity partition with only two integers is possible. The graphs $G$ for which there is a matrix in $\mathcal{S}(G)$ with partitions $[n-2,2]$ have been characterized. We find families of graphs $G$ for which there is a matrix in $\mathcal{S}(G)$ with multiplicity partition $[n-k,k]$ for $k\geq 2$. We focus on generalizations of the complete multipartite graphs. We provide some methods to construct families of graphs with given multiplicity partitions starting from smaller such graphs. We also give constructions for graphs with  matrix in $\mathcal{S}(G)$ with multiplicity partition $[n-k,k]$ to show the complexities of characterizing these graphs.
\end{abstract}

\maketitle

\section{Introduction}
Let $G$ be a simple graph on $n$ vertices and consider the set  $\mathcal{S}(G)$ of all possible real-valued weighted symmetric adjacency matrices associated to $G$, where the diagonal entries are free in that they may or may not be zero (the restriction herein to simple graphs avoids some unnecessary confusion when stating and proving results). The notation $\lambda_i^{(n_i)}$ is used to denote the eigenvalue $\lambda_i$ with multiplicity $n_i$. In this work, we order the eigenvalues according to their corresponding multiplicities. That is, a matrix $A\in \mathcal{S}(G)$ has spectrum $\sigma(A) =\{ \lambda_1^{(n_1)}, \dots, \lambda_\ell^{(n_\ell)}\}$, where $n_1\geq \cdots\geq  n_\ell$ and the eigenvalues $\lambda_i$ are distinct. The list of eigenvalues could also be ordered according to their values, this is distinctly different from the ordering considered in this work; see \cite{KS}. 
 
 If $q(A)$ is the number of distinct eigenvalues of a symmetric matrix $A$, then for a given graph $G$, we let $q(G)=\min\{q(A): A\in \mathcal{S}(G)\}$, and refer to this parameter as the minimum number of distinct eigenvalues of $G$. It is well-known that $q(G)=1$ if and only if $G$ has no edges, and $q(G)=n$ if and only if $G$  is a path on $n$ vertices; see ~\cite{AACFMN}. Graphs with $q(G)=n-1$ are characterized in \cite{BFHHLS}. Some conditions for graphs attaining $q(G)=c$, for $c\in \{2,\dots, n-2\}$ are given in~\cite{AACFMN}. For example, if a connected graph $G$ on at least three vertices has $q(G) =2$, then $G$ cannot have a cut edge (an edge whose deletion results
in a disconnected graph), and any two non-adjacent vertices of $G$ must have at least two common neighbours (Lemma 4.2 and Corollary 4.5 of \cite{AACFMN}, respectively). Moreover, Corollary 3.6 of the same paper gives a construction for a graph on $n$ vertices satisfying $q(G)=c$ for any $1\leq c\leq n$.

For positive integers $n_1\geq \cdots\geq  n_\ell$, a partition $[n_1, n_2, \dots, n_\ell]$ of $n$ is said to be \textsl{achievable} by $G$ if there exists an $A \in \mathcal{S}(G)$ such that the spectrum of $A$ is $\{\lambda_1^{(n_1)}, \dots, \lambda_\ell^{(n_\ell)}\}$ for some set of distinct values $\lambda_i$. If $[n_1,n_2, \dots, n_\ell]$ is a partition of $n$, we use the notation $MP([n_1,n_2, \dots, n_\ell])$ to denote the set of all graphs on $n$ vertices for which the partition $[n_1, \dots, n_\ell]$ of  $n$ is achievable. 

Given a graph, some natural questions arise. For example, what multiplicity  partitions are achievable by a given graph or a family of graphs? This question is considered in~\cite{AABBCGKMW, BBFHHLSY} where all achievable partitions are listed for all graphs on fewer than 6 vertices. Another question is if a certain multiplicity partition is achievable for a graph, is it possible to characterize any other multiplicity partitions that are also achievable for the graph? In particular, for two partitions $[n_1,n_2, \dots, n_\ell]$ and $[\tilde{n}_1,\tilde{n}_2, \dots, \tilde{n}_m]$ of  $n$, when is $MP([n_1,n_2, \dots, n_\ell]) \subseteq MP([\tilde{n}_1,\tilde{n}_2, \dots, \tilde{n}_m])$? 

Another natural approach is to characterize the graphs in $MP([n_1, \dots, n_\ell])$ for some partition $[n_1, \dots, n_\ell]$. This has been answered for a limited number of partitions. For example, the main result in~\cite{MShader} is that every connected graph on $n$ vertices is in $MP([1,1, \dots, 1])$. The only graph that is in $MP([n])$ is the graph on $n$ vertices with no edges. The set $MP([n-1,1])$ is exactly the set of graphs on $n$ vertices that have one connected component that is a complete graph and the remaining components are isolated vertices (this includes the graph on $n$ vertices with no edges). The graphs in $MP([n-2,2])$ have also been exactly characterized; this characterization is given in Lemma~\ref{lem:cupvee} and see any of ~\cite{BHL, CGMJ, MS, OS1}) for a proof.

 The questions and studies mentioned above all fit under the general umbrella of the inverse eigenvalue problem (IEVP) of a graph, specifically looking at the achievable multiplicities of the eigenvalues in the IEVP. Many authors have provided partial answers to these questions; see  \cite{AACFMN,  BBFHHLSY, BFHHLS, BHPRT, KS, MS}. Here, we add to the growing body of work, focusing on joins, complete multipartite graphs and the sets $MP([n-k,k])$ for some $k$  with $k \leq \lfloor \frac{n}{2}\rfloor$.

\section{Graphs with two Distinct Eigenvalues}

Our goal is to consider the graphs $G$ with $q(G)=2$; any such graph on $n$ vertices achieves a bipartition $[n-k,k]$ for some $k=1,2,\ldots,\lfloor \frac{n}{2}\rfloor$. The initial results that we state show that this is a very large set of graphs.

The \textsl{join} of two graphs $G$ and $H$ is the graph $G\vee H$ on vertex set $V(G)\cup V(H)$, where all  edges of $G$ and $H$ are preserved, and edges are added to make every vertex of $G$ incident to every vertex of $H$. The following theorem by Monfared and Shader~\cite[Theorem 5.2]{MShader} gives a sufficient condition for the minimum number of distinct eigenvalues to be 2.

 \begin{theorem}\label{thm:join}
 Let $G$ and $H$ be two connected graphs on $n$ vertices. Then \\ $q(G \vee H)=2$.
 \end{theorem}
 
In~\cite{LOS} a large number of graphs are shown to admit only two distinct eigenvalues. Theorem~3.2 of~\cite{LOS} proves for a tree $T$ that $q(\overline{T}) =2$, unless $T$ is $P_4$, or in one of two families of trees. Further, Theorem~2.5 of the same paper proves that many bipartite graphs $G$ have the property  $q(\overline{G}) = 2$. This gives another large and diverse family of graphs with only 2 distinct eigenvalues.

 The results in~\cite{MShader} and \cite{LOS} indicate that the collection of graphs with only two distinct eigenvalues is very large and likely very difficult to characterize. Thus rather than trying to characterize all graphs $G$ with $q(G)=2$, we define the \textsl{minimal multiplicity bipartition} $MB(G)$ to be the least integer $k\leq \lfloor \frac{n}{2}\rfloor $ such that $G$ achieves the multiplicity bipartition $[n-k,k]$ (hence if $k=MB(G)$, then $G\in MP([n-k,k])$ but $G\notin MP([n-m,m])$ for any $m<k$). If $MB(G) =k$, then we say $G$ has \textsl{multiplicity bipartition $k$}. The multiplicity bipartition is only defined for graphs that admit only two distinct eigenvalues and if $G$ has $n$ vertices, then $MB(G) \leq n/2$.

It is easy to note that if $\mathcal{S}(G)$ has a matrix with only two distinct eigenvalues, then for any two distinct real numbers there exists a matrix in $\mathcal{S}(G)$ such that its spectrum consists of these two real numbers. 

For an $m\times n$ matrix $A$, the notation $A[\alpha\,|\,\beta]$ is used to denote the submatrix of $A$ lying in rows indexed by $\alpha$ and columns indexed by $\beta$. We let $J_{n \times m}$ ($J_n$) denote the $n \times m$ ($n \times n$) matrix with all entries equal to one, and we let $0_n$ denote the $n \times n$ zero matrix.

\begin{lemma}\label{lem:vectorconst}
Let $G$ be a graph on $n$ vertices with $q(G) = 2$. Then \\ $G \in  MP([n-k,k])$ if and only if there is an $A \in \mathcal{S}(G)$ with 
\[
A= u_1u_1^T+u_2u_2^T+\dots +u_ku_k^T ,
\]
where $k\leq \lfloor \frac{n}{2}\rfloor $ and $\{u_1, u_2, \dots, u_k\}$ is a set of orthonormal vectors in $\mathbb{R}^n$.
\end{lemma}
\begin{proof}
Assume that $q(G) =2$ and $G \in MP([n-k,k])$, then there is a matrix $A \in \mathcal{S}(G)$ with $\sigma(A) = \left \{   0^{(n-k)}, 1^{(k)}\right \}$. Therefore, $A$ can be written as
\begin{eqnarray*}
A=V (I_k \oplus O_{n-k}) V^T = U U^T ,
\end{eqnarray*}
where $V$ is a unitary matrix, $U=V[1, \ldots, n|1, \ldots, k]$ and $O_{n-k}$ is the $(n-k)\times (n-k)$ all zeros matrix. Let $U = (u_1, u_2, \dots , u_k)$ where $u_{i} \in \mathbb{R}^n$, $i=1, \ldots, k$. Each $u_i$ is a column of $V$, so they are orthonormal. Moreover, $A= u_1u_1^T+u_2u_2^T+\dots +u_ku_k^T$. 

Conversely, assume there is a matrix $A= u_1u_1^T+u_2u_2^T+\dots +u_ku_k^T$, where $\{u_1, u_2, \dots, u_k\}$ is a set of orthonormal vectors in $\mathbb{R}^n$. Then the spectrum for $A$ is $\left \{   0^{(n-k)}, 1^{(k)}\right \}$, so $G \in MP([n-k,k])$.
\end{proof}


We summarize the known characterizations of graphs with given minimal multiplicity bipartitions in the following lemma.  

 \begin{lemma}\label{lem:cupvee}
Let $G$ be a connected graph on $n$ vertices. Then
\begin{enumerate}
    \item $MB(G) = 1$ if and only if $G$ is the complete graph, $K_n$.

    \item  $MB(G) = 2$ if and only if 
\[
G = (K_{p_1} \cup K_{q_1}) \vee  (K_{p_2} \cup K_{q_2})\vee \dots \vee (K_{p_k} \cup K_{q_k}),
\]
for non-negative integers $p_1,\dots,p_k,q_1,\dots,q_k$ with $k>1$, and $G$ is not isomorphic to either one of a complete graph or $(K_{p_1} \cup K_{q_1}) \vee  K_{1}$.  
     \item If $MB(G) = k$, then $G$ does not have an independent set (a set of vertices for which no two are adjacent) of size $k+1$ or more.\label{lem:cupvee:coclique}
\end{enumerate}
 \end{lemma}
 
 \begin{proof}
The first statement is trivial. The second statement has appeared in~\cite{BHL, CGMJ, MS,OS1}.

The third statement is known (for example, there is a proof in~\cite{MS}), but we include a proof for completeness. From Lemma~\ref{lem:vectorconst} there is a matrix $A \in \mathcal{S}(G)$ with
\[
A= u_1u_1^T+u_2u_2^T+\dots +u_ku_k^T
\]
where $u_i$, $i=1,2,\dots,k$, are orthonormal vectors. Let $U = [u_1, u_2, \dots , u_k]$ where $u_{i} \in \mathbb{R}^n$, $i=1, \ldots, k$. If $G$ has an independent set of size $k+1$, then the rows of $U$ form $k+1$ orthogonal vectors in $\mathbb{R}^k$, which is impossible. Hence there is no independent set of size $k+1$.
\end{proof}

The following lemma indicates that for a connected graph $G$ with $q(G)=2$, if the union of the pairwise common neighbourhood of an independent set of vertices is not empty, then it cannot be too small. 

\begin{lemma}[Theorem 4.4, \cite{AACFMN}]\label{pairwisecommonneighbour}
Consider a connected graph $G$ with $q(G)=2$, and let $S$ be an independent set of vertices. If $\displaystyle \cup_{v_i,v_j\in S}(N(v_i)\cap N(v_j))$ is not empty, then \[|S| \leq |\cup_{i,j\in S}(N(v_i)\cap N(v_j))|.\]
\end{lemma}

For a given graph $G$, the minimum rank among all matrices (positive semidefinite matrices) in $\mathcal{S}(G)$ is denoted by $\mr(G)$ ($\mrr+(G)$). If a graph $G$ on $n$ vertices has $q(G)=2$ and $MB(G)=k$, from Lemma~\ref{lem:vectorconst} then there is a matrix $A \in \mathcal{S}(G)$ with spectrum $\{0^{(n-k)},1^{(k)}\}$. This implies the following lemma.

\begin{lemma}\label{lem:mmr+}
If $G$ is any graph with $q(G) = 2$, then $\mr(G) \leq \mrr+(G) \leq MB(G)$.
\end{lemma}

The parameter $\mr(G)$ has been extensively studied~\cite{FH}; and any lower bound on the minimum rank or the minimum positive semidefinite rank of a graph, is also a lower bound on the minimal multiplicity bipartition of the graph. For example, we may use the above bound on minimum rank together with \cite[Obs. 1.6]{FH} to deduce the next result.

\begin{lemma}\label{cor:easylowerbounds}
If $G$ is any graph with $MB(G) = k$, then any induced path of $G$ has length no more than $k$.
\end{lemma}

\section{Complete Multipartite Graphs}

As is standard, we use the notation $K_{p_1, p_2, \dots, p_\ell}$ for the complete $\ell$-partite graph, where $\ell$ is a positive integer and $p_1 \geq p_2 \geq \cdots \geq p_\ell$. The set of vertices is partitioned into $\ell$ disjoint parts $V_1\cup V_2\cup \cdots \cup V_{\ell}$; part $V_i$ has $p_i$ vertices for $i\in \{1, \dots, \ell\}$; no two vertices from a part are adjacent, while any two vertices from different parts are adjacent.

In this section we verify that the value of $q(K_{p_1, p_2, \dots, p_\ell})$ is either 2 or 3 assuming the graph is nonempty. One case remains unresolved namely when $p_1 \leq p_2 +\cdots + p_{\ell}$.
We also provide an upper bound for $MB(K_{p_1, p_2, \dots, p_\ell})$ in the case of $q(K_{p_1, p_2, \dots, p_\ell})=2$. In fact, when $\ell =2$ or $p_1=p_2=\cdots =p_\ell$ we show that $MB(K_{p_1, p_2, \dots, p_\ell}) = p_1$. The question of which other multiplicity partitions can be achieved by a complete multipartite graph remains open.

In an unpublished manuscript (see \cite{BFHHLS2}), it is shown that any complete multipartite graph $K_{p_1, p_2, \dots, p_\ell}$ satisfies $q(K_{p_1, p_2, \dots, p_\ell}) \leq  3$. The basic idea employed in the proof of this inequality is to note that the matrix $B=[b_{u,v}]$ with entries defined as 
\[
b_{u,v}=
\begin{cases}
0 & \text{if $u,v\in V_i$},\\
\displaystyle\frac{1}{\sqrt{p_ip_j}} & \text{if $u\in V_i$, $v\in V_j$, and $i\neq j$,}
\end{cases}
\]
satisfies $B\in \mathcal{S}(K_{p_1, p_2, \dots, p_\ell})$ and $q(B)=3$. Furthermore, it can be easily verified that the eigenvalues of $B$ are $\{ -1, 0, \ell-1\}$ with multiplicities $\ell-1, \sum_{i=1}^\ell (p_i -1), 1$, respectively. Thus if follows that $K_{p_1, p_2, \dots, p_\ell} \in MP([\;\sum_{i=1}^\ell (p_i -1), \ell-1,1])$.

There are several known results for the partitions that are achievable along with the corresponding values of $q$ for complete multipartite graphs; we list them in the following lemma. 

\begin{lemma}\label{lem:achievablePs}
For positive integers $\ell, p_i, q_i$ with $i=1,2,\ldots,\ell$

\begin{enumerate}
 \item $q(K_{p_1, p_2, \dots, p_\ell}) \in \{2, 3\}$. 
    \item $q(K_{p_1,q_1}) = \begin{cases}
                                         2 &\textrm{if $p_1=q_1$}, \\
                                         3 &\textrm{otherwise.} \\
                                      \end{cases}$
    \item If $p_1 + p_2+ \cdots + p_\ell = q_1+q_2+ \cdots +q_{\ell'}$ for $\ell, \ell' \geq 2$, then 
    \[
    q(K_{p_1,p_2, \dots, p_\ell,q_1,q_2, \dots, q_{\ell'}}) = 2.
    \]
    \item If $p_2 +\cdots + p_{\ell} < p_1$, then $q(K_{p_1, p_2, \dots, p_\ell}) = 3$. \label{complete3}
  \end{enumerate}
\end{lemma}
\begin{proof}
The first statement is from~\cite{AACFMN}. The second statement follows from \cite[Cor. 6.5]{AACFMN}.

To observe that the third statement holds, note that $p_1 + p_2+ \cdots + p_\ell = q_1+q_2+ \cdots +q_{\ell'}$ for $\ell, \ell' \geq 2$ implies that $K_{p_1,p_2, \dots, p_\ell,q_1,q_2, \dots, q_{\ell'}}$ is isomorphic to the join of $K_{p_1,p_2, \dots, p_\ell}$ and $K_{q_1,q_2, \dots, q_{\ell'}}$, and hence satisfies $q(K_{p_1,p_2, \dots, p_\ell,q_1,q_2, \dots, q_{\ell'}}) = 2$ by Theorem 2.1.

To see that the last statement holds, assume that $p_2+ \cdots + p_{\ell} < p_1$, and set $n= p_1 + p_2+ \dots + p_\ell$; this implies $n - p_1 < p_1$. If $q(K_{p_1, p_2, \dots, p_\ell}) = 2$, then $K_{p_1, p_2, \dots, p_\ell} \in MP([n-k,k])$ with $p_1 \leq k \leq n-k$ (by Statement (3) of Lemma 2.3). Hence
$p_1 \leq k \leq n-k \leq n - p_1$, implying  $2p_1 \leq n$, which is a contradiction. Hence
$q(K_{p_1, p_2, \dots, p_\ell}) \geq 3$. Finally, the equality follows from the work in the unpublished manuscript~\cite{BFHHLS2}.
\end{proof}

Continuing with the complete multipartite graph, the only case that remains unresolved is when $p_1\leq p_2+\cdots+p_{\ell}$. We suspect that in this case $q(K_{p_1, p_2, \dots, p_\ell})=2$. 

To address the specific case when all $p_i=k$, we begin by stating a technical result that is a special case of \cite[Lemma 10]{BHH} (in the notation of \cite{BHH},  we are setting $q=0$ and $p=k\geq 2$).

\begin{lemma}\label{lem:Barrett}
Let $k \geq 2$, and $M_1$ and $M_2$ be matrices that have $k$ rows and no zero columns. Then there exists a $k \times k$ matrix $R$ such that $R^T R = I_k$ and $M_{1}^{T} R M_2$ has no zero entries.
\end{lemma}

\begin{lemma}\label{qi's}
For any graph $G$ with no isolated vertices and for any number $d$ such that  $\mrr+ (G)\leq d\leq |V(G)|$, there exist vectors $q_1,\ldots, q_d$ such that $\sum_{i} q_iq_i^T\in \mathcal{S}(G)$ and each $q_i, i=1,\ldots, d$, is entry-wise nonzero.  
\end{lemma}
\begin{proof}
Since for the given parameter $d$ we have $\mrr+ (G) \leq d \leq |V|$, it follows that there is an $A \in \mathcal{S}_{+}(G)$---the set of positive semidefinite matrices in $\mathcal S(G)$---with rank$(A)=d$. Since $G$ has no isolated vertices, $A$ has no zero rows or columns. Since $A$ is positive semidefinite, $A$ can be written as $A =UU^T$, for some $n \times d$ matrix $U$. If $U = [ u_1, u_2, \ldots, u_d]$, where $u_1, u_2, \ldots, u_d$ are the columns of $U$, then we may assume that $u_1, u_2, \ldots, u_d$ are mutually orthogonal vectors in $\mathbb{R}^n$. 

Set $M_1=U^T$, and $M_2 = I_d$; each of these matrices have $d$ rows and do not have any zero columns (this follows since $A$ has no zero rows). Hence, by Lemma~\ref{lem:Barrett}, there exists an orthogonal $d \times d$ matrix $R$ such that $M_1^T RM_2 = UR$ has no zero entries. Let $Q=UR = [q_1, q_2,\ldots, q_d]$ then
\[
QQ^T =UR R^T U^T = UU^T = A \in \mathcal{S}(G).
\]
as needed.
\end{proof}

The next result is a technical result concerning a bound on the minimum semidefinite rank of joins of graphs. This result is needed to study the case when all $p_i=k$ for determining the minimum number of distinct eigenvalues of the complete multipartite graph and establish a bound on the corresponding minimal multiplicity bipartition. 
In this proof $\mathbf{1}_{s} $ is used to denote the vector in $\mathbb{R}^s$ with all entries equal to one. Similarly, $\mathbf{0}_{s}$ is used to denote the vector in $\mathbb{R}^s$ with all entries equal to zero and $O$ is the all zeros matrix, the size will be clear from context.

\begin{lemma}\label{Gjoinclique}
Consider the graph $G$ with no isolated vertices and positive integers $d, s_1, s_2, \ldots, s_d$. Assume $0<\mrr+ (G)\leq d\leq |V(G)|$. Suppose 
there exists $d$ mutually orthogonal vectors $z_1, z_2, \ldots, z_d$ such that if $Z=[z_1,z_2, \ldots, z_d]$, then $ZZ^T \in S(G)$ and $q(ZZ^T)=2$. Then $q(G\vee (K_{s_1}\cup K_{s_2}\cup\cdots \cup K_{s_d}))=2$; moreover, $MB(G\vee (K_{s_1}\cup K_{s_2}\cup\cdots \cup K_{s_d}))=d$.
\end{lemma}
\proof Let $H=G\vee (K_{s_1}\cup K_{s_2}\cup\cdots \cup K_{s_d} )$, and $s=\sum_{i=1}^d s_i$.  We construct a matrix $B\in \mathcal{S}(H)$ with $\sigma(B)=\{0,\beta\}$, where the eigenvalue $0$ has multiplicity $|V(H)|-d$ and the eigenvalue $\beta$ has multiplicity $d$. Up to translation, we may assume that the two distinct eigenvalues of $ZZ^T$ are 0, and $\lambda$. Applying Lemma~\ref{qi's}, we may replace $Z$ by $ZR$ (for some orthogonal matrix $R$) so that each column of $ZR$ is an entry-wise nonzero vector. If we set $Q=ZR$, then 
$QQ^T \in S(G)$, and $Q^TQ =R^T (Z^TZ)R = R^T (\lambda I_d) R = \lambda I_d.$ Hence the  columns of $Q$ are mutually orthogonal. Suppose the columns of $Q$ are the entry-wise nonzero vectors $q_1,\ldots, q_d$.
 Construct vectors $v_1, \ldots, v_d$ as follows
\[
v_1=\left[\begin{array}{c}
   q_1\Tstrut\Bstrut\\ \hline   \alpha_1 \mathbf{1}_{s_1} \vspace{0.1cm} \Tstrut\Bstrut\\ \hline \mathbf{0}_{s-s_1} \Tstrut\Bstrut \end{array}\right],
 v_2=\left [\begin{array}{c}  q_2 \Tstrut\Bstrut\\ \hline\mathbf{0_{s_1}} \Tstrut\Bstrut\\ \hline\alpha_2 \mathbf{1}_{s_2} \Tstrut\Bstrut\\ \hline\mathbf{0}_{s-s_1-s_2}\Tstrut\Bstrut \end{array}\right], 
\ldots,
 v_d=\left [\begin{array}{c}  q_d\Tstrut\Bstrut\\ \hline   \mathbf{0}_{s-s_d}\Tstrut\Bstrut\\ \hline \alpha_d \mathbf{1}_{s_d} \Tstrut\Bstrut \end{array}\right]
\]
the value of $\alpha_i$ will be determined. Let $q_i^T=\bmatrix q_{1,i}, &q_{2,i},& \ldots, &q_{n,i}  \endbmatrix$ for $i=1,\ldots, d$ where $|V(G)|=n$. 

For $j=2,\ldots,d+1$, let $B_{1,j}=[q_{j-1},\, q_{j-1},\, \ldots,\, q_{j-1}] $ be the $n \times s_{j-1}$ matrix with all columns equal to $q_{j-1}$. The matrix $B=\sum_{i=1}^{d} v_iv_i^T$ has the following form: 
 \[
B=\left[ 
\begin{array}{c|@{}c@{}cccc}
A\qquad && \alpha_1B_{1,2} & & \ldots &\alpha_dB_{1,d+1} \Tstrut\Bstrut\\ \hline
\alpha_1 B_{1,2}^T & & \,\alpha_1^2 J_{s_1\times s_1} &O& \ldots&O\Tstrut\Bstrut\\
                     \\ 
                 && O& \alpha_2^2 J_{s_2\times s_2} & & O \\
     \vdots              & &   & & & \vdots   \\
                       &&\vdots &  & \ddots & O\\\ 
                       \\
\alpha_d B_{1,d+1}^T &&O& \ldots& O& \alpha_d^2 J_{s_d\times s_d}\\
\end{array}\right].
\]
Now, for a positive number $\beta > \max \{||q_i||^2, i=1,\ldots, d\}$, set $\alpha_i=\displaystyle \sqrt{\frac{\beta-||q_i||^2}{s_i}}$. Then $B\in \mathcal{S}(H)$ and  $\rk(B)=d$, which implies the eigenvalue zero has multiplicity $|V(H)|-d$. On the other hand, $Bv_i=\beta v_i$ for each $i=1,2,\ldots,d$, and the vectors $v_1, \ldots, v_d$ are linearly independent. Thus $\sigma(B)=\{0,\beta\}$ with the desired multiplicities, and since $H$ has at least $d$ independent vertices, $d\leq MB(H)$. \qed 

Note that, in Lemma~\ref{Gjoinclique}, if $G$ has $\ell$ isolated vertices, then these vertices form an independent set. By Lemma 2.4, in order for $q(G)=2$, the union of the mutual common neighbours of an independent set cannot have more than $\ell$ elements; therefore, $\ell \leq \sum_i^d s_i$. Moreover, if $d=s_1=1$ and $G$ has isolated vertices, then there is a unique path from a vertex of $G$ to an isolated vertex of $G$ using the vertex of $K_{s_1}$, which implies the graph cannot have only two distinct eigenvalues. It is unclear if the statement of Lemma~\ref{Gjoinclique} holds in other cases when $G$ has $\ell$ isolated vertices.

A specific case of  Lemma~\ref{Gjoinclique} is when all the parts of a complete multipartite graph have the same size. 

\begin{corollary}
Let $H = K_{k,k,\dots, k}$, $k\geq 2$ be the complete multipartite graph on $n$ vertices. Then $q(H)=2$ and $MB(H) =k$.
\end{corollary}

\proof Using Lemma \ref{lem:achievablePs}, we know that $q(K_{k,k})=2$ and it is not difficult to determine that $MB(K_{k,k})=k$ (since the minimum semidefinite rank of $K_{k,k}$ is equal to $k$). The graph $K_{k,k}$ represents the base case for an induction argument based on the number of parts. Assume for the complete multipartite $G=K_{k,k,\dots, k}$ on $l-1$ ($l \geq 3$) parts that $q(G)=2$ with $MB(G)=k$. Let $H=K_{k,k,\dots, k}$ be the complete multipartite graph with $l$ parts. Then $H = G\vee (K_{1}\cup K_{1}\cup\cdots \cup K_{1})$, where $G = K_{k,k,\dots, k}$ with one less part than $H$. The next step involves an application of Lemma~\ref{Gjoinclique}. If we $d=k$ and using the fact that $q(G)=2$ with $MB(G)=k$, then there exists a matrix $A$ in $S(G)$ such that $A=ZZ^{T}$, where $Z$ is an $n \times d$ matrix and with $Z^{T}Z = \lambda I_{d}$. From Lemma~\ref{Gjoinclique} it follows that $q(H)=2$.
\qed

For a vertex $v$ in a graph $G$, a new graph $G'$ can be constructed by \textsl{cloning} (or \textsl{duplicating}) $v$. The graph $G'$ has vertex set $V(G')=V(G)\cup \{v'\}$ and edge set $E(G')=E(G)\cup\{v'u; u\in N[v]\}$, where $N[v]$ is the closed neighbourhood of $v$ (that is, a neighbourhood of $v$ containing $v$). It turns out that cloning a vertex of a graph $G$ with $MB(G)=k$ results in a graph $G'$ with $MB(G')\leq k$. The following proposition is proved in Theorem~6.3 of  \cite{CGMJ}, it is also implied by Corollary~4 of \cite{AABBCGKMW}. In~\cite{CGMJ}, this is used to characterize graphs $G$ with $MB(G)=2$ by constructing minimal such graphs (these are $K_1$, $K_1\cup K_1$, $K_{2,1}$, $K_{2,2,\dots,2}$ and $K_{2,2,\dots,2,1}$) and constructing all the other such graphs by cloning vertices in the minimal graphs.

\begin{proposition}\label{prop:cloning}
Let $G$ be a graph with $G \in MP([n-k,k])$. If $H$ is obtained from $G$ by cloning a vertex in $G$, then $H \in MP([n-k+1,k])$.
\end{proposition}
 
Suppose $G$ is a graph with $q(G) = 2$. If $H$ is obtained from  $G$ by cloning a vertex, then $MB(H) \leq MB(G)$. It is not clear if this inequality is ever strict. By cloning vertices in $K_{k, k, \ldots, k}$, where $k\geq 2$, we have the following consequence, reminiscent of Lemma 2.3.

\begin{corollary}\label{cor:canonical}
If $G = \bigvee_{i=1}^{l}(\cup_{j=1}^{k}K_{a_{i,j}})$ where $k\geq 2$, $i=1,2,\dots,l, j=1,2,\dots,k$, and $a_{i,j}$ are positive integers, then $q(G) =2$ and $MB(G) = k$.\qed
\end{corollary}

Unlike the case where $k=2$, for general $k$ the previous corollary does not characterize all the graphs in $MP([n-k,k])$. In Section 4, several graphs in $MP([n-k,k])$ that are not included in Corollary~\ref{cor:canonical} are given. 

\begin{theorem}\label{thm:joinpairk}
Let $G$ and $H$ be two graphs with no isolated vertices. Further assume that $q(G) = q(H)=2$ with 
$MB(G) = MB(H)$. Then $q(G \vee H) =2$ and $MB(G\vee H) \leq MB(G) \, (=MB(H))$.
\end{theorem}

\begin{proof} Assume $MB(G)=MB(H)=k$. If $k=1$, there is nothing to prove, so assume $k \geq 2$.
Let $n_1$ be the number of vertices in $G$ and $n_2$ the number of vertices in $H$. Let $A \in \mathcal{S}(G)$ be such that $\sigma(A) =\left \{ 0^{(n_1 -k)}, 1^{(k)} \right \}$ and let $B \in \mathcal{S}(H)$ be such that $\sigma(B) =\left \{ 0^{(n_2 -k)}, 1^{(k)} \right \}$. By Schur's Theorem, there exists orthogonal matrices $Q_1$ and $Q_2$ such that 
\[
A = Q_{1}^{T} (I_k \oplus O_{n_1 - k}) Q_1 \quad  \text{and} \quad 
B = Q_{2}^{T} (I_k \oplus O_{n_2 -k}) Q_2.
\]
Let 
\[
M_1 = Q_1[1, \ldots, k \mid 1, \ldots, n_1] \quad  \text{and} \quad 
M_2 = Q_2[1, \ldots, k \mid 1, \ldots, n_2],
\]
so $A = M_{1}^{T} M_1$ and $B = M_{2}^{T} M_2$.
This also implies that $M_1$ has no zero columns, since otherwise $A$ would have a row and column of zeros which would imply that $G$ would have an isolated vertex. Similarly,  $M_2$ has no zero columns. 

By Lemma~\ref{lem:Barrett}, there exists a $k \times k$ matrix $R$ such that $R^T R = I_k$ and $M_{1}^{T} R M_2$ has no zero entries. Define $C$ as follows:
\begin{eqnarray*}
 C =
\begin{bmatrix}
M_{1}^{T}\\ 
M_{2}^{T}R^{T}
\end{bmatrix}
\begin{bmatrix}
M_1 & R M_2
\end{bmatrix}
& = &
\begin{bmatrix}
M_{1}^{T}M_1 & M_{1}^{T}R M_2 \\ 
M_{2}^{T} R^T M_1 & M_{2}^{T} R^T R M_2
\end{bmatrix} \\
& = &
\begin{bmatrix}
A & M_{1}^{T}R M_2 \\ 
M_{2}^{T} R^T M_1 & B
\end{bmatrix}.
\end{eqnarray*}

Hence $C$ is positive semidefinite and $C \in \mathcal{S}(G \vee H)$ since $M_{1}^{T}R M_2$ is an entry-wise nonzero matrix. It is easy to note that $C$ has rank $k$ since $[M_1 \; RM_2]$ has a full-row rank. Therefore, $\nul(C) = n_1 + n_2 -k$. Moreover, $C^2 = 2C$ which implies $\sigma(C) = \left \{ 0^{(n_1 + n_2 -k)}, 2^{(k)} \right \}$. Hence $q(G \vee H) =2$ since $C \in \mathcal{S}(G\vee H)$ and $q(G\vee H) >1$.
\end{proof}

In Theorem~\ref{thm:joinpairk}, where $k = MB(G)=MB(H)$, if $k=1$, then the inequality $MB(G\vee H) \leq 1$ 
cannot be strict, because $MB(G\vee H)\geq 1$ for all graphs where it is defined. Similarly, if $k=2$, then the inequality $MB(G\vee H) \leq 2$ cannot be strict, because $MB(G\vee H)=1$ implies that $G\vee H$ is a complete graph with possibly isolated vertices, which is a contradiction. Moreover, if $k=3$, then the inequality $MB(G\vee H) \leq 3$ cannot be strict, because $MB(G\vee H)=1,2$ implies that $G\vee H$ is either a complete graph with possibly isolated vertices, or a graph characterized in Lemma 2.3 (2), either case is a contradiction.  

For cases of $k \geq 4$ in Theorem~\ref{thm:joinpairk}, it is still unclear if a strict inequality can hold in
$MB(G\vee H) \leq k$. We suspect, that in fact, equality among $MB(G \vee H) = MB(G) = MB(H)$ holds under the hypothesis of Theorem~\ref{thm:joinpairk}. A related matter is to determine if a version of Theorem~\ref{thm:joinpairk} still holds in the case when $MB(G)\neq MB(H)$. It turns out that the requirement of $MB(G) = MB(H)$ is essential in concluding that $q(G \vee H) =2$ as in Theorem~\ref{thm:joinpairk}. Consider the following example. Let $G=Q_6$ (the 6-dimensional hypercube), and let $S = \{x,y,z\}$ be the independent set of vertices in $G$ consisting of $x=(000000), y=(010101),$ and $z=(111111)$. Also observe that there are no common neighbours among any pair of vertices from $S$ in $G$. Let $H= P_2$. Then we have $q(G) = q(H)=2$, and $MB(G) \geq 3$ and  $MB(H) =1$. However, in the graph $G\vee H$, using the independent set $S$, it is easy to deduce that the condition of Lemma 2.4 fails. Hence $q(G \vee H) >2$.

In fact, this idea can be easily generalized as follows: Suppose $G$ is a graph with $q(G)=2$ that contains an independent set of vertices $S = \{v_1, v_2, \ldots, v_k\}$ in which 
$\displaystyle \cup_{v_i,v_j\in S}(N(v_i)\cap N(v_j)) = \emptyset$. Then for any graph $H$ with $q(H)=2$ and $|H|<k$, we have $q(G \vee H) >2$. To see this, it is enough to observe that in the graph $G \vee H$ we have 
\[  |\cup_{v_i,v_j\in S}(N(v_i)\cap N(v_j))| = |H| < |S|, \] and hence the condition of Lemma 2.4 fails to hold.

We also note that the assumption of no isolated vertices in Theorem~\ref{thm:joinpairk} is possibly a stronger condition than is in fact necessary; this assumption is used to ensure that the matrix $M_2$ in the proof has no zero columns. For instance, in the next result, which is a consequence of Lemma~\ref{Gjoinclique}, all the vertices of the second graph are isolated vertices. The proof of Lemma~\ref{lem:joinpairkisolated} is the same as the proof of Theorem~\ref{thm:joinpairk}, except that the matrix $B$ is replaced with the identity matrix. We denote the graph on $k$ vertices with no edges by $\overline{K_k}$.

\begin{lemma}\label{lem:joinpairkisolated}
Let $G$ be a graph with no isolated vertices and $q(G) =2$ with $G\in MP([n -k, k])$ for some $k \geq 2$. Then $q(G \vee \overline{K_k}) =2$ with $MB(G\vee \overline{K_k}) \leq k$.
\end{lemma}

Note that the multiplicity bipartition $[n-k,k]$ for regular complete multipartite graphs $K_{k,k,\dots, k}$ can also be obtained from the proof of Theorem~\ref{thm:joinpairk} and induction.

%

It is also interesting to note that the minimum number of distinct eigenvalues of the join of two graphs can be large.

\begin{lemma}
For any graph $G$, $q(G \vee K_1) \geq \lceil \frac{q(G)+1}{2} \rceil$.
\end{lemma}
\begin{proof}
The eigenvalues for any matrix in $\mathcal{S}(G)$ interlace the eigenvalues any matrix $\mathcal{S}(G \vee K_1)$.
\end{proof}

The next theorem  is the main result of~\cite{MShader}.

\begin{theorem}[Theorem 4.3~\cite{MShader}]\label{thm:MS}
Let $G$ be a connected graph on $n$ vertices and let
$\lambda_1, \lambda_2, \dots, \lambda_n$ be distinct real numbers. Then there exists a real symmetric matrix $A \in \mathcal{S}(G)$ with eigenvalues $\lambda_1, \lambda_2, \dots, \lambda_n$  such that none of the eigenvectors of $A$ has a zero entry.
\end{theorem}

\begin{lemma}\label{lem:MohammadMethod1}
Let $G$ be a connected graph on $n\geq 2$ vertices. Then
$q(G \vee \overline{K_n})=2$ and $MB(G \vee \overline{K_n}) = n$. 
\end{lemma}
\begin{proof}
Since $G$ is a connected graph, by Theorem~\ref{thm:MS} there exists a matrix $A \in \mathcal{S}(G)$ with positive distinct eigenvalues 
$$\lambda_1 > \lambda_2 > \cdots > \lambda_n$$
and corresponding entry-wise nonzero unit eigenvectors $v_1, \ldots, v_n$ such that
\[ A  = V^T \Lambda V= U^T U, \]
where $U = \Lambda^{\frac{1}{2}} V$, $\Lambda = \operatorname{diag}(\lambda_1, \ldots, \lambda_n)$, and $V^T=[v_1, \ldots, v_n]$. The rows of $U$ are orthogonal since $V$ is unitary. Let $C$ be the  $n \times 2n$ matrix
$
C = \begin{bmatrix}
D & U
\end{bmatrix},
$
where 
\[
D = 
\begin{bmatrix}
a_1 & 0 & \dots & 0   \\ 
0 & a_2 & \dots & 0   \\ 
\vdots && \ddots & \vdots\\
0 & 0 &\dots & a_n    \\ 
\end{bmatrix},
\] where the scalars $a_i$ $(i=1,2,\ldots, n)$ are to be determined.
Since $U$ is an entry-wise nonzero matrix, if each $a_i$ is also nonzero, then 
\[
C^T C = \begin{bmatrix}
D^2 & DU\\
U^T D & U^T U
\end{bmatrix} \in \mathcal{S}(G \vee \overline{K_n}).
\]
Further, the rows of $C$ are orthogonal and so 
\[
C C^T = 
\begin{bmatrix}
\alpha_1 & 0 & \dots & 0   \\ 
0 & \alpha_2 & \dots & 0   \\ 
\vdots && \ddots & \vdots\\
0 & 0 &\dots & \alpha_n    \\ 
\end{bmatrix},
\]
where $\alpha_i = a_{i}^{2} + \lambda_{i}^2$, $i=1, \ldots, n$. Therefore, the eigenvalues of $C C^T$ are $\alpha_i$, $i=1, \ldots, n$. The values  $a_i$ can be set so that they are all strictly positive, and $\alpha_i$ for all $i=1, \ldots, n$  are equal to some $\lambda_0 > \lambda_{1}^{2}$. Then the spectrum of  $C^T C$ is $0$ with multiplicity $n$, and $\lambda_0$ also with multiplicity $n$. This implies that $q(G \vee \overline{K_n})=2$ and $MB(G \vee \overline{K_n}) \leq n$. Finally, the vertices in $\overline{K_n}$ form an independent set of size $n$, and so  by Statement (3) of Lemma 2.3, it follows that $MB(G \vee \overline{K_n}) = n$.
\end{proof}

The same proof can be used to prove the following result.

\begin{lemma}\label{lem:MohammadMethod2}
Let $G$ be a connected graph on $n\geq 2$ vertices. Then
$q(G \vee \overline{K}_{n-1})=2$ and $MB(G \vee \overline{K}_{n-1}) = n-1$. 
\end{lemma}
\begin{proof}
As in the proof of the previous lemma, there exists a matrix $A \in \mathcal{S}(G)$ with eigenvalues
$$\lambda_1 > \lambda_2 > \cdots > \lambda_{n-1} > \lambda_n =0$$
and corresponding entry-wise nonzero unit eigenvectors $v_1, \ldots, v_n$ such that
\begin{eqnarray*}
A &= & V^T \Lambda V=U^T U,
\end{eqnarray*}
where $V^T=[v_1, \ldots, v_n]$, $\Lambda = \operatorname{diag}(\lambda_1, \ldots, \lambda_{n-1}, 0)$, $U = D V[1, \ldots, n-1 \mid 1, \ldots, n]$, and $D = \operatorname{diag}(\sqrt{\lambda_1}, \ldots, \sqrt{\lambda_{n-1}})$. Hence the rows of $U$ are orthogonal since $V$ is unitary. Let $C$ be the   $(n-1) \times (2n-1)$ matrix $C = \begin{bmatrix}
D & U
\end{bmatrix}$. 
Since $U$ is an entry-wise nonzero matrix, as long as each $a_i$ is also nonzero, the matrix 
\[
C^T C = \begin{bmatrix}
D^2 & D U\\
U^T D & U^T U
\end{bmatrix} \in \mathcal{S}(G \vee \overline{K}_{n-1}).
\]
Further, the rows of $C$ are orthogonal and so 
\[
C C^T = 
\begin{bmatrix}
\alpha_1 & 0 & \dots & 0   \\ 
0 & \alpha_2 & \dots & 0   \\ 
\vdots  & & \ddots & \vdots\\
0 & 0 &\dots & \alpha_{n-1}    \\ 
\end{bmatrix},
\]
where $\alpha_i = a_{i}^{2} + \lambda_{i}^2$, $i=1, \ldots, n-1$. Similar to the proof of the previous lemma, the spectrum of $C^T C$ is $0$ with multiplicity $n$, and $\lambda_0$ with multiplicity $n-1$. This implies that $q(G \vee \overline{K_n})=2$ and $MB(G \vee \overline{K}_{n-1}) = n-1$.
\end{proof}

\section{Constructions}\label{sec:construction}

Corollary~\ref{cor:canonical} provides an infinite family of graphs in $MP([n-k,k])$ for various values of $n$ and $k$. Corollary~\ref{cor:canonical} gives a complete characterization of graphs with $MP([n-k,k])$ for 
$k=2$, but not for any larger value of $k$. 
In this section, we consider graphs that are in $MP([n-k,k])$ but not covered in Corollary~\ref{cor:canonical}. First, we consider a direct construction of matrices corresponding to some of the graphs in Corollary~\ref{cor:canonical}. 

\begin{example}\label{cor:construction3}
Let
\[
G = (K_{a_1} \cup K_{b_1} \cup K_{c_1}) \vee (K_{a_2} \cup K_{b_2} \cup K_{c_2}) 
\]
with $a_i,b_i,c_i>1$, so that $q(G) = 2$ and $MB(G) = k$ by  Corollary~\ref{cor:canonical}. For $i=1,2$ and
$t_i \in \mathbb{R}$ to be determined, set 
\begin{align*}
v_{1,i}&=  \big[
        \underbrace{ \frac{t_i}{\sqrt{a_i}}, \dots,  \frac{t_i}{\sqrt{a_i}}}_{a_i \textrm{ times }}, \;
        \underbrace{ \frac{1}{\sqrt{b_i}},  \dots, \frac{1}{\sqrt{b_i}} }_{b_i \textrm{ times }},  
        \underbrace{ \frac{-t_i}{\sqrt{c_i}(t_i+1)}, \dots ,  \frac{-t_i}{\sqrt{c_i}(t_i+1)}}_{c_i \textrm{ times }}  
         \big]^T \\
v_{2,i}&=  \big[ 
        \underbrace{ \frac{-t_i}{\sqrt{a_i}(t_i+1)}, \dots ,  \frac{-t_i}{\sqrt{a_i}(t_i+1)}}_{a_i \textrm{ times }}, 
        \underbrace{ \frac{t_i}{\sqrt{b_i}}, \dots,  \frac{t_i}{\sqrt{b_i}}}_{b_i \textrm{ times }}, \;
        \underbrace{ \frac{1}{\sqrt{c_i}},  \dots, \frac{1}{\sqrt{c_i}} }_{c_i \textrm{ times }}  
         \big]^T \\
v_{3,i}&=  \big[ 
        \underbrace{ \frac{1}{\sqrt{a_i}},  \dots, \frac{1}{\sqrt{a_i}} }_{a_i \textrm{ times }},  
        \underbrace{ \frac{-t_i}{\sqrt{b_i}(t_i+1)}, \dots ,  \frac{-t_i}{\sqrt{b_i}(t_i+1)}}_{b_i \textrm{ times }}, 
        \underbrace{ \frac{t_i}{\sqrt{c_i}}, \dots,  \frac{t_i}{\sqrt{c_i}}}_{c_i \textrm{ times }} \;
         \big]^T. \\
\end{align*}
 Clearly 
\[
\| v_{1,i} \| =\| v_{2,i} \| =\| v_{3,i} \| = t_i^2+1+\frac{t_i^2}{(t_i+1)^2}
\]
and the three vectors are pairwise orthogonal for each $i$.

Now, form three overall  vectors by concatenation, so
\[
v_1 = \begin{bmatrix}v_{1,1} \\ v_{1,2}\end{bmatrix}, \quad v_2 = \begin{bmatrix}v_{2,1} \\ v_{2,2}\end{bmatrix}, \quad v_3= \begin{bmatrix}v_{3,1} \\ v_{3,2}\end{bmatrix}. 
\]
These three vectors all have the same norm and are pairwise orthogonal. Let $A = I - 2(v_1 v_1^T+v_2 v_2^T+ v_3 v_3^T)$.

Let $x,y$ be two vertices in $G$, then the $(x,y)$-entry of $A$ is given by
\begin{align*}
[A]_{x,y} =
\begin{cases}
\frac{1}{a_i} \left( t_i^2+1+\frac{t_i^2}{(t_i+1)^2} \right) & \textrm{ if } x \in V(K_{a_i}),\textrm{ and } y \in V(K_{a_i}); \\
\frac{1}{b_i} \left( t_i^2+1+\frac{t_i^2}{(t_i+1)^2} \right) & \textrm{ if }  x \in V(K_{b_i}),\textrm{ and } y \in V(K_{b_i}); \\
\frac{1}{c_i} \left( t_i^2+1+\frac{t_i^2}{(t_i+1)^2} \right)  &\textrm{ if }  x \in V(K_{c_i}),\textrm{ and } y \in V(K_{c_i}). \\
\end{cases}
\end{align*}

If $x$ and $y$ are both in $K_{a_i} \cup K_{b_i} \cup K_{c_i}$, but not both in the same clique (induced complete graph), then $[A]_{x,y} = 0$.

Assume that $x \in K_{a_1}$, then
\begin{align*}
[A]_{x,y} =
\begin{cases}
\frac{1}{\sqrt{a_1}\sqrt{a_2}} \left( t_1t_2+\frac{t_1}{(t_1+1)}\frac{t_2}{(t_2+1)} + 1 \right) 
  & \textrm{ if } y \in V(K_{a_2}); \\
\frac{1}{\sqrt{a_1}\sqrt{b_2}} \left( t_1 +\frac{-t_1t_2}{(t_1+1)} + \frac{-t_2}{t_2+1} \right) 
  & \textrm{ if } y \in V(K_{b_2}); \\
\frac{1}{\sqrt{a_1}\sqrt{c_2}} \left( \frac{-t_1t_2}{t_2+1} +\frac{-t_1}{(t_1+1)} + t_2 \right)
  & \textrm{ if } y \in V(K_{c_2}). \\
\end{cases}
\end{align*}
If $t_1$ and $t_2$ are distinct and positive, these are all
nonzero. Similarly we can show that for any $x \in K_{a_1} \cup
K_{b_1} \cup K_{c_1}$ and any $y \in K_{a_2} \cup K_{b_2} \cup
K_{c_2}$ that $[A]_{x,y}$ is not equal to zero. Therefore,
$A \in \mathcal{S}(G).$

\end{example}

\begin{theorem}\label{thm:generalmat}
  If there are $\ell$ matrices $M_i$ of order $k$ with the
  following properties
\begin{enumerate}
\item the rows of $M_i$ are orthogonal, and
\item for $i \neq j$ all the entries of $M_i^TM_j$ are nonzero,
\end{enumerate}
then the graph
\[
G = \bigvee_{j=1}^\ell( \cup_{i=1}^k K_{a_{i,j}}) 
\]
has multiplicity bipartition $[n-k,k]$ provided all $a_{i,j} \geq 1$.
\end{theorem}

\begin{proof}
 For $h=1,2,\dots,k$, construct vectors $v_h$ as follows:  $v_h =[v_{h,1},
  v_{h,2}, \dots, v_{h,\ell}]^T$, where each vector $v_{h,j}$ represents the
  vertices in a $\cup_{i=1}^k K_{a_{i,j}}$. The vertices in $K_{a_{i,j}}$ all receive the same value,
  namely $\frac{1}{\sqrt{a_{i,j}}}[M_j]_{h,i}$. Following the approach in Lemma~\ref{cor:construction3}, set $V= [v_1, \ldots, v_k]$, then $M=V^TV \in \mathcal{S}(G)$ and the spectrum of $M$ is $\{0^{(n-k)}, \|v_i\|^{(k)}\}$.  
\end{proof}

In the following, we give examples of matrices that satisfy the conditions of  Theorem~\ref{thm:generalmat} with order less than or equal to five, where $t>1$.
Such a  matrix of order 2  is \[
M_t
=\left[ \begin{matrix}
1 & t \\
-t & 1
\end{matrix} \right].
\]
Similarly, for orders 3, 4, and 5 we have, 
\[
M_t
=\left[\begin{matrix}
t & 1 & \displaystyle-\frac{t}{t+1} \\
\displaystyle-\frac{t}{t+1} & t & 1  \\
1 & \displaystyle-\frac{t}{t+1} & t \\
\end{matrix}\right],
\]
\[
M_t
= \left[ \begin{matrix}
t & -1 & \displaystyle\frac{1}{t} & 1\\
1 & t & -1&  \displaystyle\frac{1}{t}\\
\displaystyle \frac{1}{t} & 1 & t& -1 \\
-1 & \displaystyle \frac{1}{t} & 1 & t \\
\end{matrix}\right],
\]
\[
M_s
=\left[\begin{matrix}
p & -1 & s & r & 1\\
1 & p & -1 & s & r \\
r & 1 & p & -1 & s \\
s & r & 1 & p & -1 \\
-1 & s & r & 1 & p \\
\end{matrix} \right],
\]
with $p=\frac{-s^2+s-1}{s^2+2s}$ and $r = \frac{s}{s+1}$.

Note that in the construction in Lemma~\ref{cor:construction3}, the numerator
in the entries in the vectors come from the entries of the $3\times 3$ matrix $M_t$. This
method can easily be generalized.

The  following Lemma provides graphs with $MB(G) =3$ that are not listed in Corollary~\ref{cor:canonical}. 
\begin{lemma}\label{lem:const1}
  Let $H=K_{\alpha,\alpha}$ be a complete bipartite graph with $2\alpha \leq a_1$. Suppose 
\[
G= ( (K_{a_1}\backslash{H}) \cup K_{b_1}) \vee (K_{a_2} \cup K_{b_2}) \vee \dots \vee (K_{a_\ell} \cup K_{b_\ell}),
\]
with $a_i,b_i > 2$ and $\ell\geq 2$. Then $q(G) =2$ and $MB(G) =3$.
\end{lemma}
\proof Let $v_1=[v_{1,1}, v_{1,2},\dots, v_{1,l}]^T$, $v_2=[v_{2,1}, v_{2,2}, \dots, v_{2,l}]^T$, and 
$v_3=[v_{3,1}, v_{3,2}, \dots, v_{3,l}]^T$, where
%
\begin{eqnarray*}
v_{1,i}&=&\bmatrix \smash[b]{\underbrace{\begin{matrix}1, &   \mkern-11mu\dots, &  \mkern-11mu1,\end{matrix}}_{a_i\text{ times}}} &  \smash[b]{\underbrace{\begin{matrix} \mkern-11mu-\displaystyle \sqrt{\frac{a_i}{b_i}}w_i, &  \mkern-11mu\dots, &  \mkern-11mu-\displaystyle \sqrt{\frac{a_i}{b_i}}w_i\end{matrix}}_{b_i\text{ times}}}\endbmatrix^T\\\\\\
v_{2,i}&=&\bmatrix \smash[b]{\underbrace{\begin{matrix}w_i, &  \mkern-11mu\dots, &  \mkern-11muw_i,\end{matrix}}_{a_i\text{ times}}} &  \smash[b]{\underbrace{\begin{matrix} \mkern-11mu\displaystyle \sqrt{\frac{a_i}{b_i}}, & \mkern-11mu\dots, & \displaystyle  \mkern-11mu\sqrt{\frac{a_i}{b_i}}\end{matrix}}_{b_i\text{ times}}}\endbmatrix^T\\\\\\
v_{3,1}&=&\bmatrix 
\smash[b]{\underbrace{\begin{matrix}\beta, & \mkern-11mu\dots, \beta,\end{matrix}}_{\alpha\text{ times}}} &  \smash[b]{\underbrace{\begin{matrix}-\beta, & \mkern-11mu\dots, -\beta,\end{matrix}}_{\alpha\text{ times}}} &  \smash[b]{\underbrace{\begin{matrix} 0, &\dots, & 0\end{matrix}}_{a_1-2\alpha\text{ times}}} 
\endbmatrix^T\\\\
v_{3,2}&=&\bmatrix a, & -a ,& 0,& \dots, & 0\endbmatrix^T,
\end{eqnarray*}
where $\beta = \sqrt{1+w_1^2}$ and all vectors $v_{3,i}, i\geq 3$ are zero vectors. For $i>2$, if we choose the value of $w_i$   large enough so that $\sum_{i=2}^{\ell} a_i(1+w_i^2)-\alpha (1+w_1^2) >0$, then setting $a$ so that
\[
a^2=\displaystyle \frac{\sum_{i=1}^{\ell} a_i(1+w_i^2)-2\alpha (1+w_1^2)}{2}
\] 
results in a vector $v_3$ that has the same norm as vectors $v_1$ and $v_2$. The vectors $v_1^T, v_2^T,$ and $v_3^T$ form orthogonal rows of a $3\times |V(H)|$ matrix $U$, where the orthogonality of the  columns of $U$ represents the edges and non-edges of $H$. Thus, $U^TU\in \mathcal{S}(G)$, which completes the proof. \qed

This method can be extended to the graphs covered by Theorem~\ref{thm:generalmat}.

Using a similar approach, we can also remove edges across the join operation. Define an operation 
$ ( K_{a_1} \cup K_{b_1}) \dot{\vee} (K_{a_2} \cup K_{b_2}) $ that is the join of $K_{a_1} \cup K_{b_1}$ and $K_{a_2} \cup K_{b_2}$ with two disjoint edges removed between $K_{a_1}$ and $K_{a_2}$.

\begin{lemma}\label{lem:const2}
Suppose $G= ( K_{a_1} \cup K_{b_1}) \dot{\vee} (K_{a_2} \cup K_{b_2}) \vee  \dots \vee (K_{a_\ell} \cup K_{b_\ell})$, 
with $a_i \geq 2$  for $i\geq 1$, $b_2 \geq2$, and $\ell\geq 2$. Then $q(G) =2$ and $MB(G) = 3$.
\end{lemma}
\proof
Use the same vectors $v_{1,i}$ and $v_{2,i}$ as in the proof of Lemma~\ref{lem:const1}. Set 
\[
v_{3,1} =\bmatrix 
\begin{matrix} \beta, & -\beta, & 0,  \dots, 0, & 0,  \dots, 0\end{matrix} 
\endbmatrix^T,
\]
where there are $a_1-2$ zeros followed by $b_1$ zeros,
and 
\[
v_{3,2}  =\bmatrix 
\begin{matrix} \frac{(1+w_1w_2)}{\beta},  -\frac{(1+w_1w_2)}{\beta}, & 0, &\mkern-11mu\dots, 0,  \end{matrix} &  
\begin{matrix}\frac{(1+w_1w_2)}{\beta},  -\frac{(1+w_1w_2)}{\beta}, & 0, & \mkern-11mu\dots, 0, \end{matrix} \endbmatrix^T,
\]
where there are $a_2-2$ zeros in the first group and $b_2-2$ zeros in the second group,
and let $v_{3,i} = \bf{0}$ for $i = 3,\dots,\ell$. Let $v_3$ be the vector formed by concatenating $v_{3,i}$ for $i =1,\dots,\ell$.

The norm of $v_3$ is $2\beta^2 + 4 \frac{(1+w_1w_2)^2}{\beta^2}$. This is a continuous function in $\beta^2$ and it takes values in the interval $((4+2\sqrt{2})\sqrt{1+w_1w_2},\infty)$. The norm of $v_1$ is at least $4+2w_1^2+2w_2^2$.
Since it is possible to choose $w_1$ and $w_2$ so that the norm of $v_1$ is larger than $(4+2\sqrt{2})\sqrt{1+w_1w_2}$, it is also possible to choose $\beta$ so that the norm of $v_3$ equals $\|v_1\|$.
\qed

This method can also be extended to the graph in Theorem~\ref{thm:generalmat}.  

From Theorem~\ref{thm:join} we know that $q(P_n \vee P_n)=2$ and Lemma \ref{cor:easylowerbounds} implies that $MB(P_n \vee P_n)\geq n-1$.
We consider a related graph that achieves this same lower bound. Let $P_n^2$ be the graph on $n$ vertices labeled by $1,2,\dots, n,1',2',\dots,n'$.
Vertices $i$ and $i'$ are adjacent for all $i\in \{1,\dots,n\}$.
If $i\in {2,\dots, n-1}$, then $i$ and $i'$ are adjacent to vertices $i-1,(i-1)',i+1,(i+1)'$.
Vertices $1$ and $1'$ are adjacent to vertices $2$ and $2'$. Vertices $n$ and $n'$ are adjacent to $n-1$ and $(n-1)'$. Note that $P_n^2$ is the strong product of $P_n$ and $P_2$.

\begin{lemma}
For any $n$, $q(P_n^2)=2$ and $MB(P_n^2) = n-1$.
\end{lemma}
\begin{proof} Order the vertices of $P_n^2$ by $(1,1',2,2',\dots, n,n')$.
  For $i \in \{0,\dots,n-2\}$ let $u_i$ be the vector with the $(2i+1)$ and $(2i+2)$-entries equal to $1$, the $2i+3$ entry equal to $2$ and the $2i+4$ entry equal to $-2$., and all remaining entries set to zero. These vectors satisfy the conditions of Lemma~\ref{lem:vectorconst}, so $q(P_n^2)=2$ and $MB(P_n^2) \leq n-1$. The result follows since $P_n^2$ has an induced path of length $n-1$ and Lemma~\ref{cor:easylowerbounds}. 
\end{proof}

A graph is a \textsl{path of cliques} if its set of vertices can be partitioned into clusters, such that each cluster is a clique of size at least two, and the cliques form a path. A path of cliques whose clusters have at least two vertices can be obtained from $P_n^2$ by cloning vertices. The next result follows from Proposition~\ref{prop:cloning}.

\begin{corollary}
If $G$ is a path of cliques of size at least 4 with $k$ the longest induced path in $G$, then $q(G) = 2$ and $MB(G) = k-1$.
\end{corollary}

\section{Open Problems}

In~\cite{LOS} a large number of graphs are shown to admit only two distinct eigenvalues.
In fact, they prove that many bipartite graphs $G$ have the property that $q(\overline{G}) = 2$. This gives  another large and diverse family of graphs with only 2 distinct eigenvalues, and for all of these graph it is interesting to consider the multiplicity bipartition. This family includes the complements of many trees, in particular they show that $q(\overline{P_n})=2$ if $n\geq 6$. The only results we have are that $MB(\overline{P_6}) = MB(\overline{P_7}) =3$.

\begin{question}
What is the multiplicity bipartition for the complement of a path on at least $8$ vertices?
\end{question}

The graphs in $MP([n-2,2])$ have been exactly characterized (see~\cite{BHL, CGMJ,MS,OS1}), from this characterization it can be seen that there are no trees $T$ with $MB(\overline{T}) = 2$. 

 One of the types of trees considered in \cite{LOS} are denoted by $S_{m,n}^r$ (these are called \textsl{type-one trees}).  The graph $S_{m,n}^r$ is formed by taking a path on $r$ vertices and adding $m$ leaves to one end point and $n$ leaves to the other end point. Alternately, these trees are formed by taking $K_{1,k} \cup K_{1,\ell}$ and added one additional edge to make the graph connected.  If the edge is added between two leaves in $K_{1,k} \cup K_{1,\ell}$, then the resulting graph is $S_{k-1,\ell-1}^4$; if the edge is added between a leaf and a non-leaf then the resulting graph is either $S_{k,\ell-1}^3$ or $S_{k-1,\ell}^3$; finally, if the edge is added between two non-leaves the resulting graph is $S_{k,\ell}^2$.

Note that if $T=S^k_{m,n}$ with $k=2,3,4$, then $\overline{T}$ is formed by taking $(K_1 \cup K_{m'}) \vee (K_1 \cup K_{n'})$
and removing a single edge across the join. In Lemma~\ref{lem:const2}, it is shown that in some cases two edges can be removed across the join. We conjecture that it is also possible to remove a single edge across the join in many cases and achieve the multiplicity bipartition $[n-3,3]$.

\begin{conjecture}
Let $T=S^k_{m,n}$ with $k=2,3,4$ and $m,n>1$. Then $q(\overline{T})=2$ and $MB(\overline{T}) = 3$.
\end{conjecture}

Corollary~\ref{cor:canonical} gives many graphs in $MP([n-k,k])$, but it is not a characterization. Is it possible to determining the minimal (in terms of the vertex cloning) graphs in $MP([n-k,k])$, and then develop a characterization of the graphs in $MP([n-k,k])$?

Lemmas 4.2 and 4.3 of \cite{OS2} show that graphs $K_{2,2,\dots,2}$ and $K_{2,2,\dots,2,1}$ can achieve any multiplicity bipartition $[n-k,k]$ for $k=2,\ldots ,\lfloor \frac{n}{2}\rfloor$. Since all graphs $G$ with $MB(G)=2$ can be obtained from $K_{2,2,\dots,2}$ and $K_{2,2,\dots,2,1}$ by cloning vertices, and since cloning can preserve the multiplicity bipartition, this implies that if a graph achieves the multiplicity bipartition $[n-2,2]$, then it also achieves the multiplicity bipartition $[n-k,k]$ for $k>2$. Therefore, we have the following result.

\begin{lemma}
For any $n$, $MP([n-2,2])\subset MP([n-k,k])$ for $k \in \{2,\dots, \lfloor n/2 \rfloor\}$. 
\end{lemma}

This raises the open question of whether or not $MP([n-k,k])\subset MP([n-k-1,k+1])$ for larger values of $k$.  Characterizing graphs with $MP([n-k,k])$ for $k>2$ (or even graphs $G$ with $MB(G)= 3$) would answer this question partially but this is likely a harder question. We also suspect that it is possible to generalize Lemma~\ref{Gjoinclique} to show that a graph $G\vee (K_{s_1}\cup K_{s_2}\cup\cdots \cup K_{s_d})$ (with the conditions stated in the Lemma) is in $MP([n-k,k])$ for all $k \geq d$. This leads to our next conjecture.

\begin{conjecture}
The complete multipartite graph $K_{k,k,\ldots,k}$ can achieve all multiplicity partitions except $[n-i,i]$ for $i<k$.
\end{conjecture}

Theorem~\ref{thm:joinpairk} proves for two graphs $G$ and $H$ with $q(G) =q(H)=2$ and $MB(G) = MB(H)$, that $MB(G\vee H)\leq MB(G)$. We conjecture that $MB(G\vee H)= MB(G)$ and that this holds in a more general setting.
\begin{conjecture}
Let $G$ and $H$ be two graphs such that $q(G) = q(H) =2$ with $MB(G) =k_1$ and $MB(H) = k_2$, and $\mrr+(G) = k_1$ and $\mrr+(H) = k_2$. Then $q(G \vee H) =2$ with $MB(G\vee H) = k$, where $k = \max \left\{k_1, k_2\right\}$.  
\end{conjecture}

Lemma~\ref{lem:mmr+} states that for any graph $G$ with $q(G) = 2$, we have $ \mrr+(G) \leq MB(G)$, and currently we are not aware of any examples of such graphs in which this inequality is strict. However, we expect that these two graph parameters may differ in general for graphs that can achieve two distinct eigenvalues.
\begin{question}
Does there exist a graph $G$ with $MB(G) > \mrr+(G)$?
\end{question}

Finally we would like to consider the family of strongly regular graphs.  Any strongly regular graph $G$ has only three distinct eigenvalues, so $q(G) \leq 3$. Our final question is the following.
\begin{question}
Which strongly regular graphs (other than $K_n$, $K_{n,n}$, and their complements) admit only two distinct eigenvalues? Among the strongly regular graphs $G$ that admit only two distinct eigenvalues, determine $MB(G)$.
\end{question}
A strongly regular graph has parameters $(n,k;a,c)$ where $n$ is the number of vertices in the graph, and each vertex has degree $k$. The number of common vertices in the neighbourhoods of two adjacent vertices is $a$, and non-adjacent vertices is $c$. If a strongly regular graph $G$ has $c=1$, then by Lemma~\ref{pairwisecommonneighbour}, $q(G) = 3$.

\section*{Acknowledgments}
The work in this paper was a joint project of the
Discrete Mathematics Research Group at the University of Regina, attended by
all of the authors. Dr.~Adm's research was supported by the German Academic Exchange Service (DAAD) with funds from the German Federal Ministry of Education and Research (BMBF) and the People Programme (Marie Curie Actions) of the European Union's Seventh Frame-work Programme (FP7/2007-2013) under REA grant agreement No.605728 (P.R.I.M.E. - Postdoctoral Researchers International Mobility Experience) during his delegation to the University of Regina and work at University of Konstanz and revised during his work at Palestine Polytechnic University.
Dr.~Fallat's research was supported in part by NSERC Discovery Research Grants, Application Nos.: RGPIN-2014-06036 and RGPIN-2019-03934. Dr.~Meagher's research was supported in part by an NSERC Discovery Research Grant, Application No.: RGPIN-03952-2018. Dr.~Nasserasr's research was supported in part by an NSERC Discovery Research Grant, Application No.: RGPIN-2019-05275. Dr.~Plosker's research was supported by NSERC Discovery Grant number 1174582, the Canada Foundation for Innovation (CFI) grant number 35711, and the Canada Research Chairs (CRC) Program grant number 231250. Dr.~Yang's research was supported in part by an NSERC Discovery Research Grant, Application No.: RGPIN-2018-06800.

We wish to thank our colleagues Rupert H.~Levene, Polona Oblak, and Helena ~\v{S}migoc for pointing out 
an error with Lemma 3.4 from our original version of this paper.

\bibliographystyle{plain}

\begin{thebibliography}{99}

\bibitem{AABBCGKMW} J.~Ahn, C.~Alar, B.~Bjorkman, S.~Butler, J.~Carlson, A.~Goodnight, H.~Knox, C.~Monroek, and M.C.~Wigal, Ordered multiplicity inverse eigenvalue problem for graphs on six vertices, {arXiv preprint arXiv:1708.02438} 2017.

\bibitem{AACFMN} B.~Ahmadi, F.~Alinaghipour,  M.S.~Cavers, S.~Fallat, K.~Meagher, and S.~Nasserasr,
  Minimum number of distinct eigenvalues of graphs, {\em Electron.\ J.~Linear Algebra}, 26 (2013) pp.~673--691.

\bibitem{BBFHHLSY} W.~Barrett, S.~Butler, S.M.~Fallat, H.T.~Hall, L.~Hogben, J.C.-H.~Lin, B.L.~Shader, and M.~Young, The inverse eigenvalue problem of a graph: Multiplicities and minors, {\em J. Combin. Theory Ser. B.} (2019), {\tt https://doi.org/10.1016/j.jctb.2019.10.005}.

\bibitem{BFHHLS} W.~Barrett, S.~Fallat, H.T.~Hall, L.~Hogben, J.C.-H.~Lin, and B.L.~Shader, Generalizations of the Strong Arnold Property and the Minimum Number of Distinct Eigenvalues of a Graph, {\em Electronic Journal of Combinatorics} 24 (2) (2017) pp.~2--40. 

\bibitem{BFHHLS2} W.~Barrett, S.~Fallat, H.T.~Hall, L.~Hogben, J.C.-H.~Lin, and B.L.~Shader, Low values of $q(G)$.~To be submitted.

\bibitem{BHH} W.~Barrett, H.T.~Hall, and H.~van der Holst
  The inertia set of the join of graphs, {\em Linear
    Algebra and its Applications}, 434 (2011) pp.~2197--2203.
    
\bibitem{BHL} W.~Barrett,  H.~van der Holst,   and R.~Loewy,
  Graphs whose minimal rank is two, {\em Electron.\ J.~Linear Algebra}, 11 (2004) pp.~258--280.  
  
\bibitem{BHPRT} B.~Bjorkman, L.~Hogben, S.~Ponce, C.~Reinhart, and T.~Tranel,
  Applications of analysis to the determination of the minimum number of distinct eigenvalues of a graph,
	{\em Pure Appl. Func. Anal.}, 3 (2018) pp.~537--563.

  
\bibitem{FH} S.M.~Fallat and L.~Hogben,
  The minimum rank of symmetric matrices described by a graph: a survey, {\em Linear Algebra and its Applications}, 426 (2007) pp.~558--582.

\bibitem{KS} I.-J.~Kim  and B.L.~Shader,
  Smith normal form and acyclic matrices, {\em Journal of Algebraic Combinatorics}, 29 (2009) pp.~63--80.
  
\bibitem{LOS}  R.H.~Levene, P.~Oblak, and H.~\v{S}migoc, A Nordhaus-Gaddum conjecture for the minimum number of distinct eigenvalues of a graph, {\em Linear Algebra and its Applications}, 564 (2019), pp.~236--263.

\bibitem{CGMJ} Z.~Chen, M.~Grimm, P.~McMichael, and C.R.~Johnson,
  Undirected graphs of Hermitian matrices that admit only two distinct eigenvalues, {\em Linear Algebra and its Applications}, 458 (2014) pp.~403--428.

\bibitem{MS} K.~Meagher   and I.~Sciriha,
  Graphs that have a weighted adjacency matrix with spectrum $\{\lambda_1^{n-2},\lambda_1^2\}$, {\em arXiv preprint arXiv:1504.04178}, 2015.

\bibitem{MShader} K.H.~Monfared   and B.L.~Shader,
  The nowhere-zero eigenbasis problem for a graph, {\em Linear Algebra and its Applications}, 458 (2016) pp.~296--312.

\bibitem{OS1} P.~Oblak and H.~\v{S}migoc,
  Graphs that allow all the eigenvalue multiplicities to be even, {\em Linear Algebra and its Applications}, 454 (2014) pp.~72--90.

\bibitem{OS2} P.~Oblak and H.~\v{S}migoc,
  The maximum of the minimal multiplicity of eigenvalues of symmetric matrices whose pattern is constrained by a graph, {\em Linear Algebra and its Applications}, 512 (2017) pp.~48--70.

\bibitem{P} T.~Peters, Positive semidefinite maximum nullity and zero forcing number,
{\em Electronic Journal of Linear Algebra}, 23 (2012) pp.~815--830.
   
\end{thebibliography}

\end{document}